\documentclass{amsart}

\usepackage[latin1]{inputenc}
\usepackage[english]{babel}
\usepackage{indentfirst}
\usepackage{amssymb}
\usepackage{amsthm}
\usepackage{xcolor}
\usepackage[all]{xy}
\usepackage{hyperref}

\newtheorem{theorem}{Theorem}[section]
\newtheorem{proposition}[theorem]{Proposition}
\newtheorem{corollary}[theorem]{Corollary}
\newtheorem{lemma}[theorem]{Lemma}

\theoremstyle{definition}
\newtheorem{definition}[theorem]{Definition}
\newtheorem{example}[theorem]{Example}
\newtheorem{remark}[theorem]{Remark}

\begin{document}
\title[$p$-regular operators and weights]{$p$-regularity and weights for operators between $L^p$-spaces}



\author{E. A. S\'{a}nchez P\'{e}rez}
\address{E. A. S\'{a}nchez P\'{e}rez\\ Instituto Universitario de Matem\'atica Pura y Aplicada\\  Universitat Polit\`ecnica de Val\`encia\\ 46022 Valencia. Spain.}
\email{easancpe@mat.upv.es}

\author[P. Tradacete]{P. Tradacete}
\address{P. Tradacete\\Department of Mathematics\\ Universidad Carlos III de Madrid\\ 28911, Legan\'es, Madrid, Spain.}
\email{ptradace@math.uc3m.es}

\thanks{E. A. S\'anchez P\'erez gratefully acknowledges support of Spanish Ministerio de Econom\'{\i}a, Industria y Competitividad through grant  MTM2016-77054-C2-1-P.
P. Tradacete gratefully acknowledges support of Spanish Ministerio de Econom\'{\i}a, Industria y Competitividad through grants MTM2016-76808-P and MTM2016-75196-P}

\maketitle

\begin{abstract}
We explore the connection between $p$-regular operators on Banach function spaces and weighted $p$-estimates. In particular, our results focus on the following problem. Given finite measure spaces $\mu$ and $\nu$, let $T$ be an operator defined from a Banach function space $X(\nu)$ and taking values on $L^p (v d \mu)$ for $v$ in certain family of weights $V\subset L^1(\mu)_+$: we analyze the existence of a bounded family of weights $W\subset L^1(\nu)_+$ such that for every $v\in V$ there is $w \in W$ in such a way that $T:L^p(w d \nu) \to L^p(v d \mu)$ is continuous
uniformly on $V$. A condition for the existence of such a family is given in terms of $p$-regularity of the integration map associated to a certain vector measure induced by the operator $T$.
\end{abstract}

\section{Introduction}

The study of boundedness of relevant  operators defined between weighted $L^p$-spaces is a classical problem in harmonic analysis. Following the founding paper by M. A. Ariño and B. Muckenhoupt \cite{AM}, in recent years, a lot of attention has been devoted to understanding the precise relation between weighted $p$-estimates and the Hardy-Littlewood maximal operator \cite{BS, CS}, or Calderón-Zygmund operators \cite{Hy,HP}. On the other hand, weighted norm inequalities for linear operators are connected with a plethora of problems related to abstract constructions on summability on Banach function spaces.  Although there is a lot of classical developments and a lot of research which has been done since its publication, the monograph \cite{garu} by J. Garc\'{\i}a Cuerva and J. L. Rubio de Francia can be considered as a milestone in this classical topic.

Actually, partly motivated by \cite[Chapter VI, Theorem 6.7]{garu}, in this paper we investigate the  problem of existence of a bounded ``conjugate" family of weights $W$ for a given set of weights $V$ with respect to a given linear operator $T$, in the following sense: \textit{ given $v \in V$, there exists $w \in W$ in such a way that  $T$ is well defined and continuous (with uniform norm) as an operator $T:L^p(w d\nu) \to L^p(v d \mu)$.} We will provide conditions for the existence of such a conjugate family in terms of $p$-regularity of certain operator associated to $T$ (see Section \ref{existence}). Moreover, in the case that the existence of such a family is not possible, we also analyze some average domination that may be a successful substitute in some cases. This latter fact will be analyzed by means of lattice $p$-summing operators.

Let us recall that for a Banach space $E$ and a Banach lattice $Y$, an operator $T:E \to Y$ is \emph{lattice $p$-summing} if there is a constant $C>0$ such that for every $\{x_i\}_{i=1}^n\subset E$,
$$
\bigg\|\bigg(\sum_{i=1}^n|Tx_i|^p\bigg)^{\frac1p}\bigg\|_Y \leq C \, \sup_{x^* \in B_{E^*}} \bigg(\sum_{i=1}^n|\langle x_i, x^* \rangle|^p \bigg)^{\frac1p}.
$$

Lattice $p$-summing operators provide an analogue of the well-known class of $p$-summing operators which is better suited to the order and lattice properties of the spaces involved. This class of operators was introduced for the case $p=1$ by L. P. Yanowskii and later studied by J. Szulga, N. J. Nielsen, N. Danet and B. P. Pomares for the general case during the eighties (see
\cite{dan,nie,pom1,pom2,szu1,szu2,yan}).

Recall also  that given quasi-Banach lattices $E,F$, and $0< p< \infty,$ a linear operator $T:E\rightarrow F$ is said to be \emph{$p$-regular} if there is a constant $K>0$ such that for every $\{x_i\}_{i=1}^n\subset E$
$$
\bigg\|\bigg(\sum_{i=1}^n|Tx_i|^p\bigg)^{\frac1p}\bigg\|\leq K \bigg\|\bigg(\sum_{i=1}^n|x_i|^p\bigg)^{\frac1p}\bigg\|.
$$

This family of operators, which can be thought of as a generalization of the notion of regular operator ---i.e. those which can be written as a difference of two positive operators---, was first considered in \cite{B} in connection with the interpolation theory of Banach lattices (and also implicitly in \cite{N}). In particular, these operators are related to the so-called Marcinkiewicz-Zygmund inequalities \cite{deju} and other interpolation properties (see \cite{RT2,STpqregular} for recent developments about this).

The structure of the paper is as follows. First, in Section \ref{prelim}, we analyze some basic properties of lattice $p$-summing  operators and $p$-regular operators between Banach lattices, introducing the elements that will be necessary later on, and showing the main separation argument. In Section \ref{s:pregular}, we study the factorization properties of $p$-regular operators and as an application we provide a proof of classical results of W. B. Johnson, L. Jones \cite{JJ} and L. Weis \cite{W} which allow us to consider certain operators on Banach function spaces as operators on weighted $L^p$ spaces. The corresponding results for lattice $p$-summing operators are collected in Section \ref{s:latticep}. Finally, Section \ref{existence} is devoted to the study of existence of conjugate families of weights. The techniques in this last section exploit both $p$-regular and lattice $p$-summing operators together with vector measure technology.

For unexplained terms and notions concerning Banach lattices, Banach spaces and operator theory we refer the reader to the monographs \cite{DJT,LT2,libro}.

\section{Preliminaries} \label{prelim}

Let $(\Omega,\Sigma,\mu)$ be a  finite measure space. Let  $L^0(\mu)$  be the space of all measurable real functions on $\Omega$, identifying functions that are equal $\mu$-a.e.   We say that a Banach function space over $\mu$ is a Banach space $X(\mu) \subseteq L^0(\mu)$ that contains all simple functions,   and whenever $|f|\leq |g|$ and $g\in X(\mu)$, then $f\in X(\mu)$ with $\|f\|_{X(\mu)}\leq \|g\|_{X(\mu)}$. For the aim of simplicity, we will write sometimes $X$ instead of $X(\mu)$ if the measure does not play a relevant role or is clear in the context. 
Recall $X(\mu)$ is order continuous if for every sequence $(f_n)_n\in X(\mu)$, $f_n\downarrow 0$ implies $\|f_n\|_{X(\mu)}\to 0$. In this case,
the K\"othe dual
$$
X(\mu)'=\{g \in L^0(\mu): fg \in L^1(\mu) \textrm{ whenever } f\in X(\mu)\},
$$
equipped with the norm
$$
\|g\|_{X(\mu)'}:=\sup_{\|f\|_{X(\mu)}\leq1} \, \int fg \,d\mu,
$$
coincides with the (topological) dual space $X(\mu)^*$.

A Banach function space $X(\mu)$ is $p$-convex for $p\ge 1$  if there is $K>0$ such that for every finite set of functions $f_1,...,f_n \in X(\mu),$
$$
\Big\| \Big( \sum_{i=1}^n \big| f_i \big|^p \Big)^{1/p} \Big\| \le
K \Big( \sum_{i=1}^n \big\| f_i \big\|^p \Big)^{1/p}.
$$
The infimum of all such constants is called the $p$-convexity constant and is denoted by $M_{(p)}(X(\mu)).$

For $0 < p < \infty$, the $p$-th power  of $X(\mu)$  is
defined as the linear space
$$
X(\mu)_{[p]}:=\{f \in L^0(\mu): |f|^{1/p} \in X(\mu)\}.
$$
If $X(\mu)$ is $p$-convex (with $p$-convexity constant  $M^{(p)}(X(\mu))$ equal to $1$), then $X(\mu)_{[p]}$ is  a Banach function space over $\mu$  with norm given by
$$
\|f\|_{X(\mu)_{[p]}}:= \| |f|^{1/p} \|_{X(\mu)}^p.
$$
It follows that $X(\mu)_{[p]}$ is order continuous if and only if so is $X(\mu)$, and in this case simple functions are dense in both spaces (see \cite[Ch.2]{libro} and \cite{LT2} for this and further details on Banach function spaces). Moreover, a functional  $x' \in (X(\mu)_{[p]})^*$ acts on  $|x|^p \in X(\mu)_{[p]} $ (where $x \in X(\mu)$) as
$$
\langle x',|x|^p \rangle=\int x' |x|^pd\mu.
$$

The operator counterpart of the definition of $p$-th power of a Banach function space is given by the so called $p$-th power factorable operators. Namely, given a Banach space $E$, an operator $T:X(\mu) \to E$ is $p$-th power factorable if for every $f \in X(\mu)$, we have that
$$
\| T(f) \|_E \le K \, \| |f|^{1/p} \|_{X(\mu)}^p
$$
(see \cite[Ch.5]{libro} for the main properties of this class of operators).

Let us also recall some preliminaries on vector measures: Given a $\sigma$-algebra of subsets of some set $\Omega$ and a Banach space $E$, a function $m:\Sigma \rightarrow E$ is a (countably additive) vector measure if
$$
m(\cup_{n=1}^{\infty}A_n)=\sum_{n=1}^{\infty}m(A_n)
$$
in the norm topology of $E$ for all sequences $\{A_n\}_n$ of pairwise disjoint sets of $\Sigma$. For each $x^*$ in the dual space $E^*$, the expression
$$
\langle m,x^* \rangle(A):=\langle m(A),x^*\rangle, \quad A\in\Sigma,
$$
defines a scalar measure. A set $A \in \Sigma$ is $m$-null if $m(B) = 0 $ for every $B\in\Sigma$ with  $B \subseteq A.$  Two (scalar or vector) measures are said to be equivalent if they have the same null sets.

A Rybakov measure of $m$ is a scalar measure of the form  $\langle m,x^* \rangle$ for some $x^* \in E^*$ that is equivalent to $m$. It is well-known that every vector measure has a Rybakov measure (see \cite[Ch.IX]{dies}).

The space of integrable functions with respect to the vector measure $m$, which we denote $L^1(m)$, is the space of all measurable (scalar valued) functions $f$ that are integrable with
respect to each scalar measure $\langle m,x^* \rangle$ for $x^*\in E^*$ and
for each $A \in \Sigma$ there exists $\int_A f dm \in E$
with
$$
\Big\langle\int_A f dm,x^*\Big\rangle=\int_A f d\langle m,x^* \rangle,\quad \text{for }x^* \in X^*.
$$
The space $L^1(m)$  is an order continuous Banach function space over any Rybakov measure with the norm
$$
\|f \|_{L^1(m)}:= \sup_{x^* \in B_{E^*}} \int |f | \, d | \langle m, x^* \rangle |
$$
and the $m-$a.e. order. The spaces $L^p(m)$ for $1 \le p<\infty$ are defined in the obvious way; in terms of $p$-th powers,
we have that $L^p(m)= (L^1(m))_{[1/p]}$, or equivalently, $L^p(m)_{[p]} = L^1(m).$ We point out that the spaces $L^p(m)$ are always $p$-convex with $M_{(p)}(L^p(m))=1.$

As usual, we will write $I_m:L^1(m) \to E$ for the integration map
$$
I_m(f):= \int f d m,\quad \text{for }\,f \in L^1(m).
$$

An operator $T:E \to F$ between Banach spaces $E$ and $F$ is called $p$-summing if there is a constant $K$ such that for every finite set $x_1,...,x_n \in E,$
$$
\Big( \sum_{i=1}^n \big\| T(x_i) \big\|^p \Big)^{1/p}  \le
K \, \sup_{x^* \in B_{E^*}} \Big( \sum_{i=1}^n \big|\langle x_i, x^* \rangle \big|^p \Big)^{1/p}.
$$
We refer the reader to the monograph \cite{DJT}, for background on ideals of operators between Banach spaces (in particular, on $p$-summing and $p$-dominated operators).

Let us fix in what follows the context of the present paper. We will use systematically the two classes of operators cited in the Introduction.
The following section describes the separation result concerning this inequality that will be used several times through all the paper.

\subsection{Weighted norm inequalities and domination for $p$-regular type operators} \label{sbs}

Consider an operator $T$ from a Banach space $E$ to a Banach lattice $Y$. Given $1\leq p<\infty$, let $\frac1p+\frac1{p'}=1$. Consider the following two inequalities:

\begin{itemize}

\item
$T$ is \textit{lattice $p$-summing } if and only if
for each finite set $x_1, \ldots, x_n \in E$,
\begin{equation} \label{in1}
\bigg\|  \Big( \sum_{i=1}^n |T(x_i)|^p  \Big)^{1/p} \bigg\| \le C \,\sup_{(a_i) \in B_{\ell^{p'}} }\bigg\| \sum_{i=1}^n a_i x_i \bigg\|.
\end{equation}
The lattice $p$-summing norm of $T$ is the least possible constant $C>0$ appearing in this inequality and will be denoted by $\lambda_p(T)$.

\item If $E$ is also a Banach lattice, $T$ is
\textit{$p$-regular} if and only if for each finite set $x_1, \ldots, x_n \in E$,
\begin{equation} \label{in2}
\bigg\|  \Big( \sum_{i=1}^n |T(x_i)|^p  \Big)^{1/p} \bigg\| \le C \,\bigg\|  \Big( \sum_{i=1}^n |x_i|^p  \Big)^{1/p} \bigg\|
\end{equation}
The $p$-regular norm of $T$ is the least possible $C>0$ appearing in this inequality and will be denoted by $\rho_p(T)$.

\end{itemize}

Note that for $x_1, \ldots, x_n$  in a Banach lattice we always have
$$
\sup_{(a_i) \in B_{\ell^{p'}} }\bigg\| \sum_{i=1}^n a_i x_i \bigg\|\leq \bigg\| \sup_{(a_i) \in B_{\ell^{p'}} }\sum_{i=1}^n a_i x_i \bigg\|= \bigg\|  \Big( \sum_{i=1}^n |x_i|^p  \Big)^{1/p} \bigg\|.
$$
In particular, every lattice $p$-summing operator is $p$-regular with $\rho_p(T)\leq\lambda_p(T)$. However, it is easy to check that the identity on $L_p$ is $p$-regular but not lattice $p$-summing. In fact, every positive operator between Banach lattices, or even a difference of positive operators, is $p$-regular for every $p\ge 1$.

Note that the above inequalities can be written in terms of an expression involving the norm of an element in the dual of a certain Banach space. This observation, together with a standard separation argument yields the following:

\begin{proposition}\label{kyfan:reg}
Let $X$, $Y$ be $p$-convex Banach function spaces (with $p$-convexity constant 1). An operator $T:X\to Y$ is \textit{$p$-regular} if and only if there exists $C>0$ such that for every $y^* \in B_{(Y_{[p]})^*}$ there is $z^* \in B_{(X_{[p]})^*}$ with
$$
|\langle |T(x)|^p,  y^* \rangle| \le C^p \, \langle |x|^p, z^* \rangle,\quad x\in X.
$$
Moreover, the smallest constant $C>0$ appearing in this inequality coincides with $\rho_p(T)$.
\end{proposition}

\begin{proof}
Suppose $T:X\to Y$ is $p$-regular. Note that since we are assuming $X$ and $Y$ have $p$-convexity constants both equal to 1, it follows that both $X_{[p]}$ and $Y_{[p]}$ are Banach lattices.

Fix $y^*\in B_{(Y_{[p]})^*}$. Let us consider the family of affine functions given as follows: for $x_1,\ldots,x_n\in X$, let
$$
f_{x_1,\ldots,x_n;y^*}(x^*):=\sum_{i=1}^n \langle |T(x_i)|^p,  y^* \rangle  - \rho_p(T)^p\sum_{i=1}^n\langle |x_i|^p,x^*\rangle,\quad\quad x^*\in B_{(X_{[p]})^*}.
$$
Clearly, these functions belong to $C(B_{(X_{[p]})^*},w^*)$ and it is easy to check they satisfy the conditions of Ky-Fan's Lemma \cite[Lemma 9.10]{DJT}. For this, just note that given $x_1,\ldots,x_n\in X$ there is always $x^*\in B_{(X_{[p]})^*}$ such that
$$
\sum_{i=1}^n\langle |x_i|^p,x^*\rangle=\Big\| \sum_{i=1}^n|x_i|^p\Big\|_{X_{[p]}}=\Big\| \Big(\sum_{i=1}^n|x_i|^p\Big)^{1/p}\Big\|_X^p.
$$
Hence, for $x_1,\ldots,x_n\in X$, such an $x^*$ yields
\begin{eqnarray*}
f_{x_1,\ldots,x_n;y^*}(x^*)&=&\sum_{i=1}^n \langle |T(x_i)|^p,  y^* \rangle  - \rho_p(T)^p\sum_{i=1}^n\langle |x_i|^p,x^*\rangle\\
&=&\sum_{i=1}^n \langle |T(x_i)|^p,  y^* \rangle  - \rho_p(T)^p\Big\| \Big(\sum_{i=1}^n|x_i|^p\Big)^{1/p}\Big\|^p\\
&\leq & \Big\|  \Big( \sum_{i=1}^n |T(x_i)|^p  \Big)^{1/p} \Big\|^p - \rho_p(T)^p\Big\| \Big(\sum_{i=1}^n|x_i|^p\Big)^{1/p}\Big\|^p\leq0.
\end{eqnarray*}

Therefore, by Ky-Fan's Lemma there exists $z^*\in B_{(X_{[p]})^*}$ such that for every $x_1,\ldots,x_n\in X$, we have
 $$
 f_{x_1,\ldots,x_n;y^*}(z^*)\leq0.
 $$
In particular, for every $x\in X$ we have
$$
|\langle |T(x)|^p,  y^* \rangle| - \rho_p(T)^p \, \langle |x|^p, z^* \rangle=f_{x;y^*}(z^*)\leq0,
$$
as claimed.

The converse implication is straightforward.
\end{proof}

Since lattice $p$-summing operators are $p$-regular, the inequality of Proposition \ref{kyfan:reg} also holds for this class. However, in a similar fashion one can show the following characterization:

\begin{proposition}\label{kyfan:lat}
Let $E$ be a Banach space and $Y$ a $p$-convex Banach function space (with $p$-convexity constant 1). An operator $T:E\to Y$ is \textit{lattice $p$-summing} if and only if there exists $C>0$ such that for every $y^* \in B_{(Y_{[p]})^*}$ there is a regular probability measure $\eta \in \mathcal M(B_{E^*})$ such that
$$
|\langle |T(x)|^p,  y^*\rangle| \le C^p \, \int_{B_{E^*}}  |\langle x, \cdot \rangle |^p  \, d \eta, \quad x \in E.
$$
Moreover, the smallest constant $C>0$ appearing in this inequality coincides with $\lambda_p(T)$.
\end{proposition}

\begin{proof}
The proof follows the same lines as that of Proposition \ref{kyfan:reg} using, for a non-zero functional $y'\in Y_{[p]}^*$ and $(x_i)_{i=1}^n\subset E$,  the family of functions given by
$$
g_{x_1,\ldots,x_n;y^*}( \eta):=  \Big( \sum_{i=1}^n \langle |T(x_i)|^p,  y^* \rangle \Big) - \int_{B_{E^*}} \sum_{i=1}^n |\langle x_i, \cdot \rangle |^p  \, d \eta,$$
for $\eta\in \mathcal M(B_{E^*})=(C(B_{E^*},w*))^*$.
\end{proof}

\begin{remark}
The separation arguments mentioned above can be also used when the w*-compact  sets are not necessarily the unit ball of a dual space. In order to see this, just note that the Hahn-Banach Theorem  separation argument works when the functions under consideration belong to a certain $C(K)$-space for some compact Hausdorff space $K$. We will use this version in the next section, where general families of weights will be considered. This will allow to deal with families of positively norming sets for Banach lattices, that have been studied and applied in a similar context in \cite[S.5]{sancheztradacetepositivelynorming}.
\end{remark}

\section{$p$-regular operators between $p$-convex Banach lattices}\label{s:pregular}

A direct application of the arguments given in the previous section gives the following characterization for $p$-regular operators between $p$-convex Banach function spaces. A version of this result can be found in Proposition 4.1 of \cite{defmasan}. It provides an example of what will be the fundamental tool used in Section \ref{conjugate}.

\begin{proposition} \label{pconvexop}
Consider an operator $T:X(\mu_1) \to Y(\mu_2)$, where $X(\mu_1)$ and   $Y(\mu_2)$ are $p$-convex order continuous Banach function spaces over $(\Omega_1, \mu_1)$ and $(\Omega_2,\mu_2)$, respectively. The following assertions are equivalent:
\begin{itemize}
\item[(i)]
$T$ is  $p$-regular.

\item[(ii)]
There is $C>0$ such that for each $0 \le y^*\in B_{(Y(\mu_2)_{[p]})^*}$ there is $0 \le z^* \in B_{(X(\mu_1)_{[p]})^*}$ such that for every $x \in X(\mu_1)$,
$$
 \Big( \int_{\Omega_2} |T(x)|^p \, y^*\, d \mu_2 \Big)^{1/p}
\le C \,
\Big( \int_{\Omega_1} |x|^p \, z^*\, d \mu_1 \Big)^{1/p}.
$$

\item[(iii)]
For each $0 \le y^* \in B_{(Y(\mu_2)_{[p]})^*}$ there exists $0 \le z^* \in B_{(X(\mu_1)_{[p]})^*}$ such that the following diagram commutes
$$
\xymatrix{
X(\mu_1) \ar[rr]^{T} \ar@{.>}[d]_{i} & &   Y(\mu_2)  \ar@{.>}[d]_{i}.\\
L^p(z^* d \mu_1)  \ar@{.>}[rr]^{\tilde T}&  & L^p(y^* d \mu_2)
}
$$
\end{itemize}
\end{proposition}

\begin{proof}
Note this can be seen as a particular case of the classical Maurey-Rosenthal Theorem, the reader can find a detailed proof in many places, for example in Theorem 3.1 of \cite{defasan}. The equivalence (i) $\Leftrightarrow$ (ii) follows from Proposition \ref{kyfan:reg}, replacing the functional $y^*\in B_{(Y(\mu_2)_{[p]})^*}$ by its modulus if necessary, and noting that the duality pairing in $(X(\mu_1)_{[p]})^*$ (respectively, in  $(Y(\mu_2)_{[p]})^*$) corresponds to
$$
 \langle |x|^p, x^* \rangle=\int_{\Omega_1}|x|^p x^* d\mu_1,\quad\quad x\in X(\mu_1),\,x^*\in (X(\mu_1)_{[p]})^*,
$$
(respectively,  $\langle |y|^p, y^* \rangle=\int_{\Omega_2}|y|^p y^* d\mu_2$ for $y\in Y(\mu_2),\,y^*\in (Y(\mu_2)_{[p]})^*$).

For the equivalence with (iii), just note that the formal inclusion $i:X(\mu_1)\to L^p(z^* d\mu_1)$ is bounded. Indeed, for $f\in X(\mu_1)$ we have
$$
\|f\|_{L^p(z^* d\mu_1)}=\Big( \int_{\Omega_1} |f|^p \, z^*\, d \mu_1 \Big)^{1/p}\leq \Big( \||f|^p\|_{X(\mu_1)_{[p]}}\|z^*\|_{(X(\mu_1)_{[p]})^*}\Big)^{1/p}\leq\|f\|_{X(\mu_1)}.
$$
Similarly, $i:Y(\mu_2)\to L^p(y^* d\mu_2)$ is also bounded.
\end{proof}

Note that, as in Proposition \ref{kyfan:reg}, the smallest possible constant $C>0$ appearing in (ii) coincides with $\rho_p(T)$.

For endomorphisms on a $p$-convex Banach lattice, the previous result can be pushed forward to get a single weight as follows:

\begin{theorem}\label{endop-reg}
Let $1\leq p<\infty$. Suppose $X$ is a $p$-convex Banach function space over some finite measure space $(\Omega,\Sigma,\mu)$. If $T:X\to X$ is $p$-regular, then there exists a strictly positive functional $g\in B_{(X_{[p]})^*}$ such that $\tilde T:L^p(g d\mu)\to L^p(g d\mu)$ is bounded with $\|\tilde T\|\leq 2\rho_p(T)$.
\end{theorem}

\begin{proof}
Let $g_0=\chi_\Omega$. Without loss of generality we can assume $g_0\in B_{(X_{[p]})^*}$. Since $T:X\to X$ is $p$-regular and $X$ is $p$-convex, by Proposition \ref{pconvexop}, there exists $g_1\in B_{(X_{[p]})^*_+}$ such that for every $f\in X$
$$
\int_\Omega|Tf|^p g_0 d\mu \le \rho_p(T)^p \int_\Omega |f|^p g_1 d\mu.
$$
Inductively, using Proposition \ref{pconvexop}, we can construct a sequence $(g_i)_{i\in\mathbb N}\subset B_{(X_{[p]})^*_+}$ such that for $f\in X$ and $i\in\mathbb N$ we have
$$
\int_\Omega|Tf|^p g_i d\mu \le \rho_p(T)^p \int_\Omega |f|^p g_{i+1} d\mu.
$$

Now, let $g=\sum_{i\in\mathbb N} 2^{-i}g_i$. This series is convergent in the norm of $X_{[p]}^*$ and we clearly have $g\in B_{(X_{[p]})^*_+}$. Therefore, for $f\in X$ we have
$$
\int_\Omega |Tf|^p g d\mu=\sum_{i\in\mathbb N}2^{-i} \int_\Omega |Tf|^p g_i d\mu \le \rho_p(T)^p \sum_{i\in\mathbb N}2^{-i} \int_\Omega |f|^p g_{i+1} d\mu \le 2\rho_p(T)^p\int_\Omega|f|^p g d\mu.
$$
Taking into account that simple functions belong to $X$ and these are dense in $L^p(g d\mu)$, the conclusion follows.
\end{proof}

As an application of this we can provide alternative proofs of some classical results given in \cite{JJ} and \cite{W}:

\begin{corollary}
Let $1\le p<\infty$ and $T:L^p(\mu)\to L^p(\mu)$ be any bounded operator. There is $g\in L^1(\mu)_+$ such that $\tilde T:L^2(gd\mu)\to L^2(gd\mu)$ is bounded.
\end{corollary}

\begin{proof}
Suppose first $p\ge 2$. In this case, $L^p(\mu)$ is 2-convex and $T:L^p(\mu)\to L^p(\mu)$ is 2-regular (by \cite[Theorem 1.f.14]{LT2}). Hence, the conclusion follows from Theorem \ref{endop-reg}. Finally, if $p<2$, then we apply the previous argument to $T^*:L^{p'}(\mu)\to L^{p'}(\mu)$ (with $\frac1p+\frac1{p'}=1$).
\end{proof}

\begin{corollary}
Suppose $X$ is a Banach function space over some finite measure space $(\Omega,\Sigma,\mu)$. If $T:X\to X$ is regular (that is, a difference of positive operators), then there exists a strictly positive function $g\in L^1(\mu)$ such that for every $1\leq p\leq\infty$, $\tilde T:L^p(g d\mu)\to L^p(g d\mu)$ is bounded.
\end{corollary}

\begin{proof}
We apply Theorem \ref{endop-reg} for $p=1$. Note $X$ is 1-convex and if $T$ is regular, then it is in particular $1$-regular. Hence, there is a strictly positive functional $g\in X^*\subset L^1(\mu)$ such that
$$
\tilde T:L^1(g d\mu)\to L^1(g d\mu)
$$
is bounded. The same argument applied to $T^*$ yields that for some $g'\in X$, the operator
$$
\tilde T^*:L^1(g' d\mu)\to L^1(g' d\mu)
$$
is also bounded, or equivalently,
$$
\tilde T:L^\infty(g d\mu)\to L^\infty(g d\mu)
$$
is bounded. By a standard interpolation argument, it follows that
$$
\tilde T:L^p(g d\mu)\to L^p(g d\mu)
$$
is bounded for every $1\leq p\leq \infty$, as claimed.
\end{proof}

\section{Lattice $p$-summing operators}\label{s:latticep}

Recall an operator $T:E \to Y$ from a Banach space $E$ to a Banach lattice $Y$ is called lattice $p$-summing for some $1\leq p<\infty$ if there is a constant $C>0$ such that for all $x_1, \ldots, x_n \in E$, we have
$$
\bigg\|  \Big( \sum_{i=1}^n |T(x_i)|^p  \Big)^{1/p} \bigg\| \le C \,\sup_{(a_i) \in B_{\ell^{p'}} }\bigg\| \sum_{i=1}^n a_i x_i \bigg\|= C\sup_{x^* \in B_{E^*} } \Big(\sum_{i=1}^n |\langle x_i,x^*\rangle|^p\Big)^{1/p}.
$$
This class of operators was introduced and analyzed in \cite{nie}.

Proposition \ref{kyfan:lat}, in the case that $Y$ is  a $p$-convex Banach function lattice, yields that $T:E\rightarrow Y$ is lattice $p$-summing if and only if
for every $y' \in (Y_{[p]})^*$ there is a regular probability measure $\eta \in \mathcal M(B_{E^*})$ such that
$$
|\langle |T(x)|^p,  y' \rangle| \le C^p \, \int_{B_{E^*}}  | \langle x, x' \rangle|^p \, d \eta(x'),  \quad x \in E.
$$

As a direct consequence we have the following factorization theorem (in the case that $Y$ is a $p$-convex space). Some related results, without assuming $p$-convexity on $E$, can be found in \cite{dan, nie}. Moreover, for the case of $Y$ being $p$-concave, a more general class of operators has been characterized in \cite[Theorem 2.5]{pala} along similar lines.

\begin{proposition} \label{qsupimproving}
Consider an operator $T:E \to Y(\mu)$, where $E$ is a Banach space and $Y(\mu)$ is a $p$-convex order continuous Banach function space over $\mu$. The following assertions are equivalent.
\begin{itemize}
\item[(i)]
$T$ is lattice $p$-summing.

\item[(ii)]
There is $C>0$ such that for each $0 \le y^* \in B_{Y^*}=B_{Y'}$ there is a Borel probability measure on $B_{E^*}$ such that for every $x_1,...,x_n \in E$,
$$
 \Big( \int_{\Omega} |T(x)|^p \, y^*\, d \mu \Big)^{1/p}
\le C \,
\Big( \int_{B_{E^*}} |\langle x, x^* \rangle|^p \, d \eta(x^*) \Big)^{1/p}.
$$
\item[(iii)]
For each $0 \le y^* \in B_{(Y_{[p]})'}$ there is a Borel probability measure on $B_{E^*}$ such that the operator $i \circ T:E \to L^p( y^* d \mu)$ factors through a subspace $S$ of $L^{p}(\eta)$ as
$$
\xymatrix{
E \ar[rr]^{T} \ar@{.>}[d]_{i} & &   Y(\mu) \ar[rr]^{i} & &  L^p( y^* d \mu), \\
 C(B_{E^*}) \ar@{.>}[rr]^{i} &  & S \ar@{.>}[urr]^{\tilde T}
& & }
$$
and the $p$-summing norm is uniformly bounded for all $0 \le y^* \in B_{(Y_{[p]})'}.$
\end{itemize}
\end{proposition}
\begin{proof}
Let us explain first  a technical relation between the K\"othe dual  $(Y(\mu)_{[p]})'$ of the $p$-th power  of a $p$-convex Banach function space $Y(\mu)$ and the space of pointwise multiplication operators $M(Y(\mu),L^p(\mu))$ that will be used in the proof below and also in some other proofs of the present paper. If $Y(\mu)$ is such a function space, we have that
$$
(Y(\mu)_{[p]})' = M(Y(\mu),L^p(\mu))_{[p]}
$$
isometrically. Direct computations give the proof; the reader can find more information about this in \cite[Chapter 2]{libro}.

The equivalence (i) $\Leftrightarrow$ (ii) is given in Proposition \ref{kyfan:lat}. For (ii) $\Rightarrow$ (iii), just take into account that for each $y^* \in  B_{(Y_{[p]})'}$, we have that the property in (ii) means that $i \circ T$ is well-defined and $p$-summing with norm that is uniform for every element in $B_{(Y_{[p]})'}$. Consequently, by Pietsch factorization theorem (cf. \cite[2.13]{DJT}), we get the factorization in (iii). Finally, if (iii) holds we directly get (ii) for a constant $C>0$ as a consequence of the factorizations and the uniform bound for the $p$-summing norms.
\end{proof}

Let us remark that, given a Banach space $E$ and a probability regular Borel measure $\eta$ on a compact Hausdorff space $K$, the formal inclusion $i:C(K) \to L^p(\eta)$ is not only the canonical $p$-summing operator, but also the
prototype of lattice $p$-summing map. This is because in an $L^p$-space, we have
$$
\Big\| \Big( \sum_{i=1}^n |x_i|^p \Big)^{1/p} \Big\|_{L^p(\eta)} = \Big( \sum_{i=1}^n \big\|x_i \big\|_{L^p(\eta)}^p \Big)^{1/p}  \le
\sup_{x^* \in B_{\mathcal M(K)}}  \Big( \sum_{i=1}^n \big| \langle x_i,x^* \rangle \big|^p \Big)^{1/p}.
$$

The following is a direct consequence of the definitions.  Let $1 \le p$ and let $1/p + 1/p' =1.$
Recall that an  operator $T:E \to F$ between Banach spaces is $p$-dominated if there is a constant $K$ such that
 every finite sets $x_1,\ldots,x_n \in E,$ $y_1^*,\ldots,y_n^* \in F^*,$
$$
\sum_{i=1}^n \big| \langle T(x_i), y_i^* \rangle  \big|  \le
K \, \sup_{x^* \in B_{E^*}} \Big( \sum_{i=1}^n \big|\langle x_i, x^* \rangle \big|^p \Big)^{1/p} \sup_{y \in B_{F}} \Big( \sum_{i=1}^n \big|\langle y, y_i^* \rangle \big|^{p'} \Big)^{1/p'}.
$$
 By Kwapie\'n's Factorization Theorem (cf. \cite[19.3]{deflo}), an operator $T$ is $p$-dominated if it can be factored as $T=R \circ S$ where $S$ is $p$-summing and $R^*$ is $p'$-summing (with $\frac1p+\frac1{p'}=1$).

\begin{remark} \label{eso}
Let $X$, $Y$ be Banach lattice and let $E$, $F$ be Banach spaces.
Let $S_0:E \to X$ be
lattice $p$-summing, $T_0:X \to Y$  $p$-regular, and $R_0: Y \to F$ is such that $R_0^*:F^* \to Y^*$ is lattice $p'$-summing.

Then $T:E \to F$ defined as $T=R_0 \circ T_0 \circ S_0$ is $p$-dominated.

\end{remark}
\begin{proof}
Using \cite[Proposition 1.d.2.(iii)]{LT2}, it follows that
\begin{eqnarray*}
\sum_{i=1}^n \langle T(x_i), y_i^* \rangle & = & \sum_{i=1}^n \langle  T_0 \circ S_0(x_i), R_0^*(y_i^*) \rangle\\
& \le & \bigg\langle\Big(\sum_{i=1}^n  |T_0 \circ S_0(x_i)|^p\Big)^{1/p}, \Big(\sum_{i=1}^n |R_0^*(y_i^*)|^{p'}\Big)^{1/p'}\bigg\rangle \\
& \le & \rho_p(T_0) \Big\| \Big( \sum_{i=1}^n |S_0(x_i)|^p \Big)^{1/p} \Big\| \, \Big\| \Big( \sum_{i=1}^n |R_0^*(y_i^*)|^{p'} \Big)^{1/p'} \Big\| \\
& \le &  \rho_p(T_0)\lambda_p(S_0)\lambda_{p'}(R_0^*) \sup_{x^* \in B_{E^*}}  \Big( \sum_{i=1}^n |\langle x_i, x^* \rangle|^p \Big)^{\frac{1}{p}} \,  \sup_{y \in B_{F}}  \Big( \sum_{i=1}^n |\langle y, y_i^* \rangle|^{p'} \Big)^{\frac{1}{p'}}.
\end{eqnarray*}
The conclusion follows.
\end{proof}

However, Remark \ref{eso} can be improved in order to give a characterization of a certain class of $p$-dominated operators.
Next result gives an equivalent statement to Kwapie\'n's Factorization Theorem for $p$-dominated operators (see \cite[19.3]{deflo}), in a Banach lattice context, for a special class of $p$-factorable operators (see \cite[18.6]{deflo}). The main difference with Kwapie\'n's result is that the factorization here is given through $L^p$-spaces, and not only through a subspace of an $L^p$-space.

\begin{corollary} \label{thqsuppreg}
Let $T:E \to F$ be an operator between Banach spaces $E$ and $F$, and $1< p < \infty$. The following assertions are equivalent.
\begin{itemize}
\item[(i)]
There exist a $p$-convex order continuous Banach function space $X(\mu)$ and a $p$-concave Banach function space $Y(\nu)$ with order continuous dual such that $T$ factors as $T=R_0 \circ T_0 \circ S_0$,
$$
\xymatrix{
E \ar[rr]^{T} \ar@{.>}[d]_{S_0} & &  F,\\
X(\mu)  \ar@{.>}[rr]^{T_0}&  & Y(\nu)
\ar@{.>}[u]_{R_0}}
$$
where $S_0$ is lattice $p$-summing, $T_0$ is $p$-regular and $R_0^*$
is lattice $p'$-summing.

\item[(ii)]
There is a factorization for $T$  as
$$
\xymatrix{
E \ar[rr]^{T} \ar@{.>}[d]_{S_1} & &   F\\
L^p(\mu)  \ar@{.>}[rr]^{T_1}&  & L^p(\nu)
\ar@{.>}[u]_{R_1}}
$$
where $S_1$ is $p$-summing and $R_1^*$ is $p'$-summing.

\end{itemize}

Thus, if the equivalent statements (i) and (ii) hold, then $T$ is both $p$-dominated and $p$-factorable.

\end{corollary}
\begin{proof}
(i) $\Rightarrow$ (ii) A standard application of the Maurey-Rosenthal Theorem provides a factorization of $T_0$ as $M_{h^{1/p}} \circ \hat T_0 \circ M_{g^{1/p}}$, where $M_{h^{1/p}}$ is a multiplication operator $M_{h^{1/p}}:X(\mu) \to L^p(\mu)$ and $M_{g^{1/p}}: Y'(\nu) \to L^p(\nu)$, that is $h \in (X(\mu)_{[p]})'$ and $g \in (Y'(\nu)_{[p]})'$ (see \cite[Theorem 3.1]{defasan}). Proposition \ref{qsupimproving} gives that the compositions $S_1=M_{h^{1/p}} \circ S_0$ and $R_1* = M_{g^{1/p}} \circ R_0^*$ are $p$-summing and $p'$-summing, respectively, and so we get (ii).

The converse is straightforward, since every operator among $L^p$-spaces is $p$-regular, and every operator into an $L^p$-space is $p$-summing if and only if it is lattice $p$-summing.
\end{proof}

\section{Conjugate weights for operators on families of weighted $L^p$-spaces} \label{existence}

Fix a finite measure $\mu$.
Suppose that we have an operator $T$ that is well defined from an ideal of $L^0(\mu)$ to $L^p(v d\mu)$,
for $v$ belonging to a given family of weights $V\subset L^1(\mu)_+$. As we said in the Introduction, we are interested in solving the following question: \textit{Is there another family of weights $W$ satisfying that for every $v \in V$ there is $w \in W$ such that $T:L^p(w d \mu) \to L^p(v d \mu)$ is well-defined and continuous with a uniform bound for $v\in V$? }

According to the classes of operators we considered in the previous sections, we will give two solutions. The first one will give an average estimate for such a domination, that will provide a factorization through a subspace of an $L^p$ space. The second one will give an exact solution to the problem, together with a description of the corresponding conjugate class of weights.

In order to do this, we will need some tools based on vector measures. The following definition can be found in \cite[S.5]{sancheztradacetepositivelynorming}.

\begin{definition}
Let $(\Omega, \Sigma, \mu)$ be a finite measure space.
Given a norm bounded subset $V\subset L^1(\mu)$, we define the (finitely additive) vector measure $m_V: \Sigma \to \ell^\infty(V),$
\begin{equation} \label{eqvecme}
m_V(A):=  \Big( \int_A v \, d \mu  \Big)_{v \in V} \in \ell^\infty(V), \quad A \in \Sigma.
\end{equation}
\end{definition}

It was shown in \cite[Lemma 5.1]{sancheztradacetepositivelynorming} that if $V$ is a relatively weakly compact set of $L^1(\mu)_+$ ---equivalently, uniformly integrable---, then $m_V$ defined above is a countably additive vector measure.

Consider the space  $L^p(m_V)$ of $p$-integrable functions with respect to $m_V$. By the definition of the vector measure $m_V$, we clearly have that
$$
\|f\|_{L^p(m_V)}= \sup_{v \in V} \|f\|_{L^p(v d \mu)}, \quad f \in L^p(m_V).
$$
Therefore, the formal inclusions
$$
L^p(m_V) \hookrightarrow L^p( v d \mu),
$$
are bounded for every $v\in V$ with norm not greater than 1.

For each weight $v \in V$, let $e_v^*\in\ell_1(V)$ denote the corresponding element of the unit vector basis, and let $\varphi_v=e_v^*\circ I_{m_V}\in L^1(m_V)^*$. That is,
$$
\varphi_v(f)=  \int_\Omega f v d\mu,\quad \text{for}\, f \in L^1(m_V).
$$
It follows that
$$
\varphi_v(f)= \int_\Omega f v d\mu = \Big\langle \int_{\Omega} f d m_V, e_v^* \Big\rangle
\le \|f\|_{L^1(m_V)} \, \|e_v^*\|_{\ell^1(V)} = \|f\|_{L^1(m_V)}.
$$

The keystone of our approach in this section will be to consider the space $L^p(m_V)$ as the natural order continuous $p$-convex Banach function space which works globally as the range for an operator $T$ that is well-defined when taking values in each $L^p( v d \mu)$ space, for $v \in V$.

\subsection{Average domination for a given class of weights}

Throughout the rest of the paper let $1 \le p < \infty$ be fixed. Suppose that $T:X(\mu_1) \to Y(\mu_2) $ is a lattice $p$-summing operator, where $X(\mu_1)$, $Y_(\mu_2)$ are order continuous Banach function lattices over $(\Omega_1,\mu_1)$ and $(\Omega_2,\mu_2)$ respectively. Assume moreover that $Y(\mu_2)$ is $p$-convex.
Using duality for the spaces $X(\mu_1)$ and $Y(\mu_2)$, Proposition \ref{kyfan:lat} yields that
for every $y' \in B_{(Y(\mu_2)_{[p]})^*}$ there is a regular probability measure $\eta \in \mathcal M(B_{X(\mu_1)^*})$ such that
$$
\Big(\int_{\Omega_2} |T(x)|^p  y' d \mu_2 \Big)^{1/p} \le \lambda_p(T) \, \Big(\int_{B_{X(\mu_1)^*}}  \Big| \int_{\Omega_1} x \, x' \, d \mu_1  \Big|^p  \, d \eta(x')\Big)^{1/p}, \quad x \in X(\mu_1).
$$

This can be rewritten in the following terms: for every $0 \le y' \in B_{(Y(\mu_2)_{[p]})^*}$ there is a regular Borel probability measure $\eta$ and a subspace $S=\{ f_x(x'):=\int_{\Omega_1} x \, x' \, d \nu: x \in X(\mu_1)\} \subset L^p(B_{X(\mu_1)^*},\eta)$ such that the following diagram commutes:
$$
\xymatrix{
X(\mu_1) \ar[rr]^{T}  \ar@{.>}[d]_{I} & &   Y(\mu_2)  \ar@{.>}[d]_{i}\\
  S \ar@{.>}[rr]^{T_0}  &  & L^p(y' d \mu_2)
}
$$

We summarize this in the following

\begin{corollary}
Let $X(\mu_1)$, $Y(\mu_2)$ be order continuous Banach function lattices such that $Y(\mu_2)$ is $p$-convex, and a bounded operator $T:X(\mu_1) \to Y(\mu_2)$. For any norm bounded set of weights $V\subset (Y(\mu_2)_{[p]})'$, the following statements are equivalent.
\begin{itemize}

\item[(i)] The operator $T$ is  lattice $p$-summing.

\item[(ii)]  For each $v \in V$, the operators $ T:X(\mu_1) \to L^p(v \, d \mu_2)$ factor through the subspace $S$ defined above, considered as a subspace of $L^p(B_{X(\mu_1)'},\eta)$ for a certain regular probability measure $\eta$.

\end{itemize}

\end{corollary}

Let us face the original problem now. For a finite measure $\mu_2$ and a set of weights $V$ that is a relatively weakly compact set in $L^1(\mu_2)_+$, consider the vector measure $m_V$ given by (\ref{eqvecme}) and the corresponding space $L^p(m_V)$.

\begin{corollary}
Let $X(\mu_1)$ be a Banach function space. Let $V\subset L^1(\mu_2)_+$ be relatively weakly compact. Let $T$ be an operator that is well-defined from $X(\mu_1)$ to each $L^p(v d \mu_2)$. The following statements are equivalent.

\begin{itemize}

\item[(i)] The operator $T: X(\mu_1) \to L^p(m_V)$ is  lattice $p$-summing.

\item[(ii)] There is a constant $C>0$ such that for every $v \in V$ there is a regular probability measure $\eta \in \mathcal M(B_{X(\mu_1)^*})$ satisfying
$$
\Big(\int_{\Omega_2} |T(x)|^p  v d \mu_2\Big)^{1/p} \le C \, \Big(\int_{B_{X(\mu_1) ^*}}  \Big| \int_{\Omega_1} x \, x' \, d \mu_1  \Big|^p  \, d \eta(x')\Big)^{1/p}, \quad x \in X(\mu_1) .
$$

\item[(iii)]  For each $v \in V$, the operator $T:X(\mu_1) \to L^p(v \, d \mu_2)$ factors through the subspace $S=\{ \langle x, \cdot \rangle: x \in X(\mu) \}\subset L^p(\eta)$ for a certain regular probability measure $\eta$ over $B_{X(\mu_1)^*}$.
\end{itemize}

\end{corollary}

\subsection{Existence of a conjugate set for a given family of weights} \label{conjugate}

In this section we improve the domination given in the previous one under the assumption of $p$-convexity of the domain space.

Let $V \subset L^1(\mu)_+$ be a relatively weakly compact set of weights. We are interested in showing the existence of $W$, a conjugate family of weights for $V$, with respect to a given operator $T:S(\nu) \to \cap_{v \in V} L^p(v d \mu)$, were $S(\nu)$ is the space of $\nu$-a.e. equal simple functions.  Essentially, we have two different situations that should be taken into account:

\begin{itemize}

\item[(A)] The operator $T$ is defined in a Banach function space $X(\nu)$ and we require that $X(\nu) \subseteq L^p(w d \nu)$ uniformly for $w \in W$. In this case, we will show that under the adequate requirements, we can find such a  set $W\subset (X(\nu)_{[q]})'$.

\item[(B)] When there is no a priori assumption on $\cap_{w \in W} L^p(w d \nu)$ being included into some Banach function space, it will be shown that in fact one can always find an explicit Banach function space $Y(\nu)$ such that $W \subset B_{Y(\nu)}$.

\end{itemize}

Let us start with the formal definition of a conjugate family of weights $W$ of a given class of weights $V$ with respect to a given operator $T:X(\nu) \to \cap_{v \in V} L^p(v d \mu)$.

\begin{definition}
Given measure spaces $(\Omega_1,\mu_1)$, $(\Omega_2,\mu_2)$, a set $V\subset L^1(\mu_2)_+$ and an operator $T: X(\mu_1) \to L^p(vd\mu_2)$ such that
$$
\sup_{v\in V} \|T\|_{X(\mu_1) \to L^p(vd\mu_2)}<\infty,
$$
we say that $W\subset L^1(\mu_1)_+$ is a \textit{conjugate family of weights} for $V$ (with respect to $T$) if for every $v\in V$ there is $w \in W$ such that
\begin{enumerate}
\item the formal identity $i_w:X(\mu_1) \hookrightarrow L^p(w d \mu_1)$ is continuous,
\item $T:L^p(w d \mu_1)\rightarrow L^p(vd\mu_2)$ is continuous,
\end{enumerate}
with
$$
\sup_{w\in W} \|i_w\|<\infty,\quad\quad\text{and}\quad\sup_{v\in V} \|T\|_{L^p(w d\mu_1) \to L^p(vd\mu_2)}<\infty,
$$
where $w$ in the last supremum is determined in terms of $v$.
\end{definition}

For the case when $T$ is just defined on the set of simple functions over $(\Omega_1,\mu_1)$ and no boundedness is assumed, then $W\subset L_1(\mu_1)$ will be required to satisfy only (2) in the above definition.

\begin{example}
Consider a probability space $(\Omega,\Sigma,\mu)$, a disjoint partition $(A_i)_{i\in\mathbb N}$ of $\Omega$ (i.e. $A_i\cap A_j=\emptyset$ for $i\neq j$ and $\bigcup_{i\in\mathbb N} A_i=\Omega$) and let $T:L^\infty(\mu) \to L^2(\mu)$ be the operator given by
$$
T(f):= \sum_{i\in\mathbb N} \Big( {\int_{A_i} f \, d\mu} \Big) \chi_{A_i}.
$$
Take any set $V \subseteq B_{L^1(\mu)_+}$.
Since $\|f\|_{L^1(\mu)}\leq \|f\|_{L^2(\mu)}$ for $f\in L^2(\mu)$, for $v \in V$ we have
\begin{eqnarray*}
\|\sum_{i\in\mathbb N}  \Big(\int_{A_i} f \, d\mu \Big) \,  \chi_{A_i} \|_{L^2(v d \mu)}& = & \Big(\sum_{i\in\mathbb N} \Big({\int_{A_i} f \, d\mu} \Big)^2 \int_{A_i}  v d \mu \Big )^{1/2}\\
&\le& \sup_{i\in\mathbb N}  \Big|{\int_{A_i} f \, d\mu} \Big| \Big( \sum_{i\in\mathbb N} \int_{A_i}  v d \mu\Big)^{1/2}\\
&\le& \| f\|_{L^2(\mu)}.
\end{eqnarray*}
Therefore, $W=\{\chi_{\Omega} \}$ is a conjugate set of weights for $V$ with respect to $T$.

\end{example}

\vspace{0.5cm}

The existence of a conjugate family of weights for a given class of weights, in the sense we have written above, does not always hold, even when the family $V$ reduces to a single weight. In general, given an operator from a Banach function space $X(\mu)$ to an $L^p$-space, it may happen that there is no a conjugate weight allowing an extension of $T$ through a weighted $L^p$-space. Let us show an example of this situation.

\begin{example}
Consider a non-atomic probability  measure $\mu$ and take $1 \le p < q \le 2.$ Let  $T:L^q(\mu) \to L^p(\mu)$ be the isomorphic embedding given by \cite[Corollary 2.f.5]{LT2}.
Clearly, this operator does not preserve any lattice structure; actually, we will later show that it cannot be $p$-regular.

Suppose that there is a $\mu$-integrable function $g$ such that we can extend the operator $T$ in such a way that it is defined and continuous from $L^p(g d \mu)$ to $L^p(\mu).$ In particular, we have $L^q(\mu) \subseteq L^p(g d\mu),$ so there is a constant $C_1>0$ such  for every $f \in L^q(\mu)$ we have
\begin{equation}\label{a}
\Big( \int_{\Omega} |f|^p g d \mu \Big)^{1/p} \le C_1 \Big( \int_{\Omega} |f|^q d \mu \Big)^{1/q}.
\end{equation}

On the other hand, since $T$ is an isomorphic embedding, there is $K_1>0$ such that
\begin{equation}\label{b}
K_1 \Big( \int_{\Omega} |f|^q d \mu \Big)^{1/q} \le \Big( \int_{\Omega} |T(f)|^p d \mu \Big)^{1/p}.
\end{equation}
Moreover, by the continuity of the extension there is $C_2$ such that for $f\in L^q(\mu),$
\begin{equation}\label{c}
\Big( \int_{\Omega} |T(f)|^p d \mu \Big)^{1/p} \le C_2 \Big( \int_{\Omega} |f|^p g d \mu \Big)^{1/p}.
\end{equation}
Therefore, putting together \eqref{a}, \eqref{b} and \eqref{c} it follows that $\Big( \int_{\Omega} |f|^q d \mu \Big)^{1/q} $ and
$\Big( \int_{\Omega} |f|^p g d \mu \Big)^{1/p}$ are equivalent expressions for all $f \in L^q(\mu)$.

In particular, this implies that for every $n\in\mathbb N$ if we take  $(A^n_i)_{i=1}^n$ a partition of $\Omega$ such that $\mu(A^n_i)=1/n$ for all $i=1,\ldots,n,$ then
$$
\frac{1}{n^{1/q}}=\Big( \int_{\Omega} |\chi_{A^n_i}|^q d \mu \Big)^{1/q}\approx \Big( \int_{\Omega} |\chi_{A^n_i}|^pg d \mu \Big)^{1/p}=\Big( \int_{A^n_i} g d \mu \Big)^{1/p}.
$$
Thus, for every $n\in \mathbb N$
$$
\int_\Omega g d \mu = \sum_{i=1}^n \int_{A^n_i} g d \mu \approx n^{1-p/q},
$$
which is a contradiction with $g\in L^1(\mu)$ since $1-p/q >0$.
\end{example}

The previous example is the reason why we are interested in giving conditions for the existence of conjugate classes of weights. We will  show next how vector measures provide a helpful tool to prove the existence of a conjugate set of weights $W$ for a uniformly integrable set of weights $V$.

\begin{theorem} \label{vm}
Let $V$ be a relatively weakly compact set in $L^1(\mu)_+$, let $m$ be a countably additive vector measure $m: \Xi \to L^p(m_V)$ and let $\nu$ be a control measure for $m$.

The following assertions are equivalent.
\begin{itemize}
\item[(i)] The integration map $I_m: L^p(m) \to L^p(m_V)$ is $p$-regular.

\item[(ii)] There exists $W \subset B_{(L^1(m))'}$ which is a conjugate family of weights for $V$ with respect to $I_m$.
\end{itemize}

\end{theorem}

Before the proof, let us note the following

\begin{remark}
It is easy to check that a finitely additive vector measure $m: \Xi \to L^p(m_V)$ is countably additive if and only for every sequence of disjoint sets
$\{A_i\}_{i=1}^\infty$ on $\Xi$,
$$
\lim_n \sup_{v \in V} \big\| m( \cup_{i=n}^\infty A_i) \big\|_{L^p(v d \mu)} =0.
$$
\end{remark}

\begin{proof}
(i) $\Rightarrow$ (ii) Since $m:\Xi \to L^p(m_V)$ is a countably additive vector measure, the integration map $I_m: L^1(m) \to L^p(m_V)$ is continuous. Note that the formal identity $L^p(m) \hookrightarrow L^1(m)$ is also continuous, and therefore $I_m:L^p(m) \to L^p(m_V)$ is bounded.

Now, since by hypothesis $I_m$ is $p$-regular, Proposition \ref{kyfan:lat} yields that, for each $v \in V$, if we consider the (positive) functional  $\varphi_v(\cdot)= \int_\Omega \cdot \, v d\mu \in (L^p(m_V)_{[p]})' =(L^1(m_V))'$, then there is $h_v \in B_{(L^p(m)_{[p]})'}= B_{(L^1(m))'}$ such that
\begin{equation}\label{eq:domina}
\varphi_v(|I_m(f)|^p)\leq \rho_p(I_m)^p h_v(|f|^p),\quad f\in L_p(m).
\end{equation}

Note that the duality pairing in the spaces $L^1(m)$ can be written in terms of the integral with respect to a Rybakov measure $|\langle m,\phi\rangle|$ for $m$. Since by hypothesis $m \ll \nu$, the same holds for the scalar measure $|\langle m,\phi\rangle|$ and by Radon-Nikodym's theorem there is $h_\phi \in L^1(\nu)$ such that for every measurable set $A$ we have
$$
|\langle m,\phi\rangle|(A)=\int_A h_\phi d \nu.
$$
Thus, \eqref{eq:domina} can be written as
$$
\Big( \int |T(f)|^p \, v d \mu \Big)^{1/p} \le \rho_p(I_m) \Big( \int |f|^p \, h  d |\langle m,\phi\rangle| \Big)^{1/p} = \rho_p(I_m) \Big( \int |f|^p \, h \, h_\phi \,  d \nu)^{1/p}, \,\,
$$
for $ f \in L^p(m)$ and  every $v\in V$.

Therefore, the set $W=B_{(L^1(m))_+'}$  gives the conjugate set of weights. Note that this can be written as $W \cdot h_\phi$ if we want to consider it as a subset of  $L^1(\nu)$ using the Rybakov measure yielding the duality.

(ii) $\Rightarrow$ (i) First note that by definition of $m_V$, the set of functionals $\{\varphi_v: v \in V\} \subset B_{(L^1(m_V))'}$ satisfy
$$
\|f\|_{L^p(m_V)} = \sup_{v \in V } \varphi_v(|f|^p)^{1/p},\quad f\in L^p(m_V).
$$
Let $W\subset B_{(L^1(m))_+'}$ be a conjugate family of weights for $V$ with respect to $I_m$. Therefore, for every $g \in L^p(m)$ we have
$$
\|I_m(g)\|_{L^p(m_V)} = \sup_{v \in V } \varphi_v(|I_m(g)|^p)^{1/p} \le C \sup_{w \in W} w(|g|^p)^{1/p}\le C \|g\|_{L^p(m)}.
$$
Therefore, the operator $I_m$ is continuous. Now, in order to check $p$-regularity, take $f_1,\ldots,f_n \in L^p(m)$. Then, for every $\varepsilon >0$, there is a weight $v \in V$ and an associated weight $w \in W$ such that
\begin{eqnarray*}
\|( \sum_{i=1}^n |I_m(f_i)|^p )^{1/p} \|_{L^p(m_V)}^p &\le & \int_\Omega  \sum_{i=1}^n |I_m(f_i)|^p v d \mu +\varepsilon\\
&=& \sum_{i=1}^n\int_\Omega  |I_m(f_i)|^p v d \mu +\varepsilon\\
&\le & C^p \sum_{i=1}^n w(|f_i|^p) +\varepsilon \\
&\le& C^p \| ( \sum_{i=1}^n |f_i|^p )^{1/p} \|_{L^p(m)}^p +\varepsilon.
\end{eqnarray*}
Thus, $I_m$ is $p$-regular.
\end{proof}

\begin{remark}
Note that the $p$-regularity condition in (i) of Theorem \ref{vm} holds trivially when the measure $m$ is positive (which is equivalent to $I_m$ being a positive operator) or whenever $p=2$ (due to Krivine's version of Grothendieck's Theorem for operators between Banach lattices \cite[Theorem 1.f.14]{LT2}, \cite{Krivine}).
\end{remark}

\begin{remark}
It should be noted that the conjugate set of weights $W$ constructed in the implication (i) $\Rightarrow$ (ii) of Theorem \ref{vm} is actually a relatively weakly compact set in $L^1(\nu)_+$. Indeed, note that $L^1(m)$ is an order continuous Banach lattice \cite[Chapter 3]{libro} containing $L^\infty(\nu)$ for any control measure $\nu$. In particular, by \cite[Theorem 2.4.2]{MN} the formal identity $L^\infty(\nu)\hookrightarrow L^1(m)$ is weakly compact. By Gantmacher's theorem the adjoint $L^1(m)'\hookrightarrow L^1(\nu)$ is also weakly compact. Therefore, $W\subset B_{L^1(m)'}$ is a relatively weakly compact subset in $L^1(\nu)$.
\end{remark}

\begin{example}
Take $p\ge 1$, a  weakly compact set $V \subseteq L^1[0,1]_+$ and a kernel function $K:[0,1] \times [0,1] \to \mathbb R^+$ such that for every $y \in [0,1],$ the function $x \mapsto K(x,y)$ is integrable with respect to Lebesgue measure $d x$. Suppose also that for every
$v \in V$ and $A \in \mathcal B([0,1]),$ the function $y \mapsto \int_{A} K(x,y) dx$ belongs to $L^p(v d y).$
Consider also the vector measure $m_V$ and the space of $p$-integrable functions $L^p(m_V).$

Under these requirements, we can define the set function $m_K: \mathcal B([0,1]) \to L^p(m_V)$ given by
$$
m_K(A)(y):= \int_{A} K(x,y) dx, \quad A \in \mathcal B([0,1]), \,\, y \in [0,1].
$$

In order to be a vector measure, countable additivity might be required.
If $(A_n)$ is a sequence of disjoint measurable sets, we have that
$$
\lim_n \sup_{v \in V} \|m_K(\underset{m \ge n}{\overset{\infty} \bigcup} A_m)\|_{L^p(m_V)}
=
\lim_n \sup_{v \in V} \Big( \int_0^1 \Big( \int_{\underset{m \ge n}{\overset{\infty} \bigcup} A_m} K(x,y) dx \Big)^p v(y) dy \Big)^{1/p}
=0.
$$

Therefore, $m_K$ is a vector measure and the kernel operator ---defined as the integration map $I_{m_K}$--- is continuous from $L^1(m_K)$ to $L^p(m_V).$ Consequently,
it can be considered as an operator from $L^p(m_K)$ to $L^p(m_V).$

Fix a Rybakov measure $\mu_0= |\langle m_K, h \rangle|,$ where $h \in B_{(L^p(m_V))'},$
and consider the duality in $L^1(m_K)$ defined by integration with respect to $\mu_0.$
This measure is clearly absolutely continuous with respect to Lebesgue measure $dy$, and so
there is a function $h_0 = d \mu_0/ dy$. As an application
of Theorem \ref{vm}, we have that there is a constant $K$ such that for every $v \in V$ there is an element $g \in B_{(L^1(m_K))'}$ such that
$$
\Big( \int_0^1 \Big( \int_0^1 f(x) \, K(x,y) dx \Big)^p v(y) dy \Big)^{1/p}
\le K \Big( \int_0^1 |f(y)|^p \, v(y) h_0(y) \, dy \Big)^{1/p}.
$$
Therefore, the family $W= B_{(L^1(m_K))'}^+$ ---the non-negative functions of the unit ball of the dual space $(L^1(m_K))'$---, is the desired conjugate set of weights for $V$ with respect to the operator defined by the kernel $K.$
\end{example}

The vector measure approach given by Theorem \ref{vm} can be extended for operators as follows. Recall that given an operator $T:X(\nu) \to E$ ---where $X(\nu)$ is order continuous or just a space of simple functions---, we can always define a vector measure $m_T$ by $m_T(A):= T(\chi_A),$  $A \in \Sigma.$

Recall that such an operator is $p$-th power factorable is for every $f \in X(\nu),$ $\|T(f)\| \le K \| |f|^{1/p} \|^p_{X(\nu)}.$

For a Banach lattice $F$, a  subset $N \subseteq B^+_{F}$ is called positively norming if  the function
$\sup_{x^* \in N} \langle |x|, x^* \rangle$ is equivalent to the norm of $F$.
A basic characterization of such sets is the following lemma, that can be found in  \cite[Lemma 2.1]{sancheztradacetepositivelynorming}.

\begin{lemma} \label{lem2.1}
Let $X(\mu)$  be an order continuous Banach function space. For
a subset $N $ of $B^+_{X^*}$, the following statements are equivalent:
\begin{itemize}
\item[(1)] $N$ is positively norming.
\item[(2)] The pointwise product $S \cdot N$ is norming, where $S$ is the set of all functions as $\chi_A- \chi_{A^c}$, $ A \in \Sigma.$
\item[(3)] The w*-closed convex hull of $S  \cdot N$ contains a ball.
\end{itemize}
\end{lemma}

\begin{theorem}  \label{pth}
Let $V$ be a relatively weakly compact set in $L^1(\mu)_+$. Let $X(\nu)$ be an order continuous Banach function space and consider an operator $T:X(\nu)\to L^p(v d\mu)$ such that $\sup_{v\in V}\|T\|_{X(\nu)\to L^p(v d\mu)}<\infty$. The following assertions are equivalent.
\begin{itemize}

\item[(i)] The operator $T:X(\nu) \to L^p(m_V)$ is well-defined, $p$-th power factorable and the associated extension $\tilde T:L^p(m_T)\to L^p(m_V)$ is $p$-regular.

\item[(ii)]
There is a positively norming set of  weights $W \subseteq B_{(L^1(m_T))'}  \subseteq K B_{(X(\nu))'}$ conjugate to $V$ with respect tot $T$.

\end{itemize}
\end{theorem}

\begin{proof}
(i) $\Rightarrow$ (ii) If $T$ is well-defined on $L^p(m_V)$, by the definition of $m_V$ we have seen already that it is well-defined into each $L^p(v d\mu)$ with a uniform bound for $v\in V$. If $T$ is also $p$-th power factorable, then we have that $X(\mu) \subseteq L^p(m_T)$ and the commutative diagram
$$
\xymatrix{
X(\nu) \ar[rr]^{T} \ar@{.>}[rd]_{i} & &   L^p(m_V) \\
& L^p(m_T)\ar@{.>}[ru]_{\tilde T} &
}
$$
Thus, there are constants $C>0$ and $C'>0$ such that
$$
\sup_{v \in V} \| T(f)\|_{L^p(v d \mu)} \le C \|f\|_{L^p(m_T)} \le C' \|f\|_{X(\mu)}, \quad f \in X(\mu).
$$
Now, considering the extended operator $\tilde T:L^p(m_T) \to L^p(m_V)$, since both spaces involved are order continuous and $p$-convex, Proposition \ref{kyfan:reg} gives the domination property that is needed: for each $v \in V$ there is $w \in B_{((L^p(m_T))_{[q]})'}= B_{(L^1(m_T))'}=W$ satisfying
$$
\int |T(f)|^p v d \mu \le \rho_p(\tilde T)^p \, \int |f|^p w d \nu, \quad f \in L^p(m_T).
$$
The density of simple functions in $L^1(\nu)$, gives the result.

(ii) $\Rightarrow$ (i) We need to prove first that $T$ is well-defined as an operator from $L^p(m_T)$ to $L^p(m_V)$. We know that it is well-defined in each component. Fix $v \in V$. Then there exists  $w \in W$ and a constant $K$ (independent of $v$ and $w$) such that
$$
\| T(f)\|_{L^p(v d \mu)} \le K \| f \|_{L^p(w d \nu)}, f \in X(\nu).
$$
Since this happens for every $v \in V$, and taking into account that $W\subset B_{(L^1(m_T))'}$, we obtain for every $f \in X(\nu)$,
$$
\|T(f)\|_{L^p(m_V)} = \sup_{v \in V} \|T(f)\|_{L^p(v d \mu)} \le K \sup_{w \in W} \|f\|_{L^p(w d\nu)} \le K
\|f\|_{L^p(m_T)}.
$$
Therefore the operator is well-defined.

Let us see now that $p$-th power factorable. Since we have that
$X(\nu) \subset L^p(w d \nu)$ for every $w \in W$ and the set of  weights is positively norming and norm bounded in $L^1(m_T)'$, we have that
$X(\nu) \subseteq L^p(m_T)$. The inclusion is continuous (the norms of the extended operators are uniformly bounded) so we have that $T$ is $p$-power factorable (Th.5.11 in \cite{libro}).

Finally, the $p$-regularity of the extension $\tilde T:L^p(m_T)\to L^p(m_V)$ follows from Theorem \ref{vm}.
\end{proof}

\begin{remark}
As in Theorem \ref{vm}, the $p$-regularity for the extension  can be removed from the statement of (ii) if $T$ is for example positive or when $p=2$.
\end{remark}

\end{document}